\documentclass[11pt]{amsart} 
\usepackage{amsmath,amssymb,a4wide} 
\usepackage{mathabx}
\usepackage{color}
\usepackage{subcaption}
\usepackage{mathrsfs}
\usepackage{tikz,pgfplots}
\newtheorem{theorem}{Theorem} 
\newtheorem{lemma}[theorem]{Lemma} 
 
\newtheorem{corollary}[theorem]{Corollary}

\theoremstyle{definition}
\newtheorem{remark}[theorem]{Remark}

\newcommand{\R}{{\mathbb R}} 
 
\newcommand{\N}{{\mathbb N}}

\newcommand{\E}{{\mathbb E}}

\newcommand{\dd}{{\rm d}}
\newcommand{\bigo}[1]{\mathcal{O}(#1)}
\newcommand{\cW}{\mathcal{W}}

\usepackage{ulem}
\DeclareMathOperator{\Tr}{Tr}

\title[]
{Drift-preserving numerical integrators for stochastic Hamiltonian systems}

\date{\today}

\author{Chuchu Chen}
               \address{State Key Laboratory of Scientific and Engineering Computing, 
               Institute of Computational Mathematics and Scientific/Engineering Computing, Academy of Mathematics and Systems Science, Chinese Academy of Sciences, 100190~Beijing, China} 
               \email{\tt chenchuchu@lsec.cc.ac.cn}
\author{David Cohen}
              \address{Department of Mathematics and Mathematical
              Statistics, Ume{\aa} University, 90187~Ume{\aa}, 
              Sweden} 
              \email{\tt david.cohen@umu.se} 
\author{Raffaele D'Ambrosio}
          \address{Department of Engineering and Computer Science and Mathematics, 
           University of L'Aquila, L'Aquila (AQ), Italy}  
          \email{\tt raffaele.dambrosio@univaq.it} 
\author{Annika Lang}
       \address{Department of Mathematical Sciences, Chalmers University of Technology \& 
        University of Gothenburg, 41296~G\"oteborg, Sweden} 
        \email{\tt annika.lang@chalmers.se}

\begin{document}


\begin{abstract}
The paper deals with numerical discretizations of separable nonlinear Hamiltonian systems with additive noise. 
For such problems, the expected value of the total energy, along the exact solution, drifts linearly with time. 
We present and analyze a time integrator having the same property for all times. Furthermore, strong and weak 
convergence of the numerical scheme along with efficient multilevel Monte Carlo estimators are studied. 
Finally, extensive numerical experiments illustrate the performance of the proposed numerical scheme. 
\end{abstract}


\maketitle
{\small\noindent 

Stochastic Hamiltonian systems. Energy. Trace formula. Numerical schemes. 
Strong convergence. Weak convergence. Multilevel Monte Carlo. 

\section{Introduction}
Hamiltonian systems are universally used as mathematical models to describe the dynamical evolution of physical systems in science and engineering. 
If the Hamiltonian is not explicitly time dependent, then its value is constant along exact trajectories of the problem. 
This constant equals the total energy of the system. 

The recent years have seen a lot of research activities in the design and numerical analysis of energy-preserving numerical integrators 
for deterministic Hamiltonian systems, see for instance 
\cite{bgi18,bi12,cos14,ch11,g96,h10,it09,k16,m14,mb16,mqr99,qm08,wws13} 
and references therein. 

This research has naturally come to the realm of stochastic differential equations (SDEs). 
Indeed, many publications on energy-preserving 
numerical integration of stochastic (canonical) Hamiltonian systems have lately appeared. 
Without being too exhaustive, we mention \cite{cch16,cd14,fl09,hzz11} 
as well as the works \cite{MR3273263,MR3576620,chj18} on invariant preserving schemes. 
Observe, that such extensions describe 
Hamiltonian motions perturbed by a multiplicative white noise in the sense of Stratonovich. 
In some sense, this respects the geometric structure of the phase space. Hence one also has 
conservation of the Hamiltonian (along any realizations). 
The derivation of energy-preserving schemes in the Stratonovich setting 
follows without too much effort from already existing 
deterministic energy-preserving schemes. This is due to the fact that 
most of the geometrical features underlying deterministic Hamiltonian
mechanics are preserved in the Stratonovich setting. 
This is not the case in the It\^o framework considered here.

An alternative to the above Stratonovich setting is to add a random term to the deterministic 
Hamiltonian in an additive way. One can then show that the expected value of the Hamiltonian (along the exact trajectories) 
now drifts linearly with time, leading to a so called trace formula, see for instance the work \cite{smh04} in the case of a linear stochastic oscillator. 
To the best of our knowledge, drift-preserving numerical schemes 
for such a problem have only been theoretically studied in the case of quadratic Hamiltonians \cite{bb12,bb14,c12,cjz17,cs12,hsw06,smh04,st15} and in the case of linear 
stochastic partial differential equations \cite{ac18,aclw16,cchs19,cls13,h08}. 
For instance, recent contributions in structure-preserving numerics for stochastic problems have addressed conservation properties of stochastic Runge--Kutta methods applied to stochastic Hamiltonian problems of It\^o type~\cite{bb14,bb12}. In particular, proper stochastic perturbations of symplectic Runge--Kutta methods have been investigated as drift-preserving schemes. However, these methods seem particularly effective only for linear problems, while, in the nonlinear case, the drift is not accurately preserved in time. Moreover, in the case of additive noise, the more the amplitude of the stochastic term increases, the more the accuracy of the drift preservation deteriorates (see, for instance, Table~1 in~\cite{bb14}). Hence, this evidence reveals a gap in the existing literature that requires \emph{ad hoc} numerical methods, effective in preserving the drift in the expected Hamiltonian also for the nonlinear case. Closing this gap is the main purpose of this article.

In the present publication, we develop and analyze drift-preserving schemes for  
stochastic separable Hamiltonian systems, not necessarily quadratic, 
driven by an additive Brownian motion. In particular, 
we propose a novel numerical scheme that exactly satisfies the 
trace formula for the expected value of the Hamiltonian for all times (Section~\ref{sect-dp}). 
In addition, under general assumptions, we show mean-square and weak convergence of 
the newly proposed numerical scheme in Sections~\ref{sect-ms}~and~\ref{sect-weak} 
to give a complete picture of its properties. 
We show that both errors converge with order~$1$ and that the weak order 
is in general not twice the strong order in this specific case. 
On top of that, in Section~\ref{sect-weak} we introduce Monte Carlo and multilevel Monte Carlo estimators 
for the given numerical scheme and derive their convergence 
properties along with their computational costs. 
The main properties of the proposed time integrators are then 
numerically illustrated in Section~\ref{sec:numerics}. 

\section{Presentation of the drift-preserving scheme}\label{sect-dp}
Let $T>0$, $d$ be a fixed positive integer, $(\Omega, \mathcal F, \mathbb P)$ be a probability space with a normal filtration 
$(\mathcal F_t)_{t\in[0,T]}$, and let $W\colon[0,T]\times\Omega\to\mathbb R^d$ be a standard 
$(\mathcal F_t)_{t\in[0,T]}$-Brownian motion with continuous sample paths on $(\Omega, \mathcal F, \mathbb P)$.

For a positive integer $m$ and a smooth potential $V\colon \R^m\to\R$, let us consider the separable Hamiltonian 
$$
H(p,q)=\frac12\sum_{j=1}^mp_j^2+V(q).
$$
Consider now the following corresponding stochastic Hamiltonian system with additive noise
\begin{equation}\label{prob}
\begin{split}
\text dq(t) &=\nabla_pH(p(t),q(t))\,\text dt=p(t)\,\text dt\\
\text dp(t) &=-\nabla_qH(p(t),q(t))\,\text dt+\Sigma\,\text dW(t)=-V'(q(t))\,\text dt+\Sigma\,\text dW(t).
\end{split}
\end{equation}

Here, the constant matrix $\Sigma\in\R^{m\times d}$ has entries denoted by $\sigma_{ij}$. 
In addition, we assume that the initial values $(p_0,q_0)$ of the above SDE have finite energy in expectation, 
i.\,e., $\E[H(p_0,q_0)] < + \infty$. 
This setting covers, for instance, the following examples: the Hamiltonian considered in \cite{bb14} (where the matrix $\Sigma$ is diagonal), 
the linear stochastic oscillator from \cite{smh04}, various Hamiltonians studied in \cite[Chap.~4]{Milstein2004}.

\begin{remark}
In general, most results from the present publication could be extended to the case of a separable Hamiltonian system with additive martingale L\'evy noise. 
Extension to the case of a multiplicative (It\^o) noise would lead to a trace formula for the energy 
that depends on~$q_t$ (when the matrix $\Sigma$ depends only on $q_t$). 
This would correspond to the extra term appearing when converting between Stratonovich and It\^o stochastic integrals. 
\end{remark}

Using It\^o's lemma, one gets the trace formula for the energy of the above problem (see for instance \cite{bb14}):
\begin{equation}\label{etrace}
\E\left[H(p(t),q(t))\right]=\E\left[H(p_0,q_0)\right]+\frac12\Tr\left(\Sigma^\top\Sigma\right)t,
\end{equation}
for all times $t>0$.

We now want to design a numerical scheme having the same property. The idea we aim to present is inspired by the 
deterministic energy-preserving schemes from the literature (see \cite{ch11,h10,qm08} and references therein), 
able to provide the exact conservation of the energy of a given mechanical system. 
Energy-preserving methods represent a follow-up of the classical results on geometric numerical integration 
relying on the employ of symplectic Runge--Kutta methods, projection methods and nearly preserving integrators 
(see the comprehensive monograph \cite{haluwa} and references therein). In the general setting of deterministic B-series methods, 
a general proof for the existence of energy-preserving methods was given in \cite{fhp04} and practical examples of methods were first developed in \cite{qm08}. 

We propose the following time integrator for problem~\eqref{prob}, 
which is shown in Theorem~\ref{thmTrace} to be a \emph{drift-preserving integrator} for all times, 
\begin{equation}\label{dp}
\begin{split}
\Psi_{n+1}&=p_n+\Sigma\Delta W_n-\frac{h}2\int_0^1V'(q_n+sh\Psi_{n+1})\,\text ds,\\
q_{n+1}&=q_n+h\Psi_{n+1},\\
p_{n+1}&=p_n+\Sigma\Delta W_n-h\int_0^1V'(q_n+sh\Psi_{n+1})\,\text ds,
\end{split}
\end{equation}
where $h>0$ is the stepsize of the numerical scheme and $\Delta W_n$ are Wiener increments. 
Observe that the deterministic integral above needs to be computed exactly (as in the deterministic setting). 
This is not a problem for polynomial potentials $V$ for instance. We also observe that, 
as it happens for deterministic energy-preserving methods, the numerical scheme \eqref{dp} only requires the evaluation 
to the vector field of problem~\eqref{prob} in selected points, without requiring additional projection steps 
for the exact conservation of the trace formula \eqref{etrace}, that would inflate the computational cost of the procedure.

We would like to remark that the proposed drift-preserving scheme is not the only possibility 
to get a time integrator that exactly satisfy 
a trace formula for the energy for all times. Another possibility could be to use a splitting strategy. 
This idea is currently under investigation in~\cite{cv20}, see also Section~\ref{sec:numerics} below.

\begin{remark}
Since the proposed numerical scheme is implicit, we present two possible ways of showing its solvability. 

To show the solvability of the drift-preserving scheme~\eqref{dp}, we use the fixed point theorem.
Assuming that $V^{\prime}(x)$ is globally Lipschitz continuous, we define the function 
$$
F(\psi)=p_n+\Sigma \Delta W_n -\frac{h}{2}\int_0^1 V^{\prime}(q_n+sh \psi)\,\text ds.
$$
The solvability of the numerical integrator \eqref{dp} is thus equivalent to showing that the 
function $F$ is a contractive mapping. Since
\begin{equation*}
F(\psi_1)-F(\psi_2)=-\frac{h}{2}\int_0^1\Big( V^{\prime}(q_n+sh\psi_1)-V^{\prime}(q_n+sh\psi_2) \Big)\,\text ds,
\end{equation*}
then
\begin{equation*}
\begin{split}
|F(\psi_1)-F(\psi_2)| &\leq \frac{h}{2}\int_0^1\Big| V^{\prime}(q_n+sh\psi_1)-V^{\prime}(q_n+sh\psi_2) \Big|\,\text ds\\
&\leq \frac{h^2}{4}L|\psi_1-\psi_2|.
\end{split}
\end{equation*}
Therefore, there exists an $h^{*}=\sqrt{\frac{4}{L}}$ such that for all $h<h^*$, $F$ is a contractive mapping.

If $V\in {\mathcal C}^2$, we could also use the implicit function theorem instead. Indeed, let us define  
\[
G(\psi,h,\Delta W_n)=\psi-p_n-\Sigma \Delta W_n +\frac{h}{2}\int_0^1 V^{\prime}(q_n+sh \psi)\,\text ds.
\]
Then the solvability of $\psi$ for sufficiently small $h$ is equivalent in showing that 
\[
\Big|\nabla_{\psi}G(\psi,h,\Delta W_n)\Big|_{h=0}\Big|\neq 0.
\]
In fact 
\[
\nabla_{\psi}G(\psi,h,\Delta W_n)=\text{Id} +\frac{h^2}{2}\int_0^1 s V^{\prime\prime}(q_n+sh\psi)\,\text ds,
\]
so that 
\[
\Big|\nabla_{\psi}G(\psi,h,\Delta W_n)\Big|_{h=0}\Big|=|\text{Id}|=1\neq 0.
\]
By the implicit function theorem, we can thus conclude that for sufficiently small $h$, $\psi$ is solvable.
\end{remark}

We next show that the numerical scheme \eqref{dp} satisfies, for all times, the same trace formula for the energy as the exact solution to the SDE \eqref{prob}, i.\,e., that it is a drift-preserving integrator.  
\begin{theorem}\label{thmTrace}
Assume that $V\in\mathcal C^1$. Consider the numerical discretization of the stochastic Hamiltonian system with additive noise \eqref{prob} by 
the drift-preserving numerical scheme \eqref{dp}. Then the expected energy satisfies the following trace formula
\begin{equation}\label{ntrace}
\E\left[H(p_n,q_n)\right]=\E\left[H(p_0,q_0)\right]+\frac12\Tr\left(\Sigma^\top\Sigma\right)t_n
\end{equation}
for all discrete times $t_n=nh$, where $n\in\mathbb N$. Observe that $h$ is an arbitrary step size which is sufficiently 
small for the implicit numerical scheme to be solvable. 
\end{theorem}

\begin{proof}
Before we start, let us for convenience introduce the notation
$$
\text{Int}=\displaystyle\int_0^1V'(q_n+sh\Psi_{n+1})\,\text ds.
$$
Using the definitions of the Hamiltonian, of the numerical scheme and a Taylor expansion for the potential $V$ yields
\begin{align*}
\E\left[H(p_{n+1},q_{n+1})\right]&=\E\left[ \frac12p_{n+1}^\top p_{n+1}+V(q_{n+1}) \right]\\
&= 
\E\left[ \frac12\left( p_{n}+\Sigma\Delta W_n-h\text{Int} \right)^\top 
        \left( p_{n}+\Sigma\Delta W_n-h\text{Int} \right)
        +V(q_n)+h\Psi_{n+1}^\top\text{Int} \right] .
\end{align*}
By the definition of the numerical scheme and properties of the Wiener increments, one obtains further 
\begin{align*}
\E\left[H(p_{n+1},q_{n+1})\right]&=\E\left[\frac12\left(p_{n}^\top p_n-hp_{n}^\top\text{Int}+\left(\Sigma\Delta W_n\right)^\top\left(\Sigma\Delta W_n\right)-h(\Sigma\Delta W_n)^\top\text{Int}
-h\text{Int}^\top p_n \right.\right.\\
&\hspace*{4em}\left.\left.-h\text{Int}^\top\Sigma\Delta W_n+{h^2}\text{Int}^\top\text{Int}\right) \right. \\
&\hspace*{3em}\left.+V(q_n)+h p_{n}^\top\text{Int}+ h(\Sigma\Delta W_n)^\top\text{Int}-\frac{h^2}2\text{Int}^\top\text{Int} \right].
\end{align*}
Due to cancellations and thanks to properties of the Wiener increments, the expression above simplifies to
\begin{align*}
\E[H(p_{n+1},q_{n+1})]&=\E\left[\frac12 p_{n}^\top p_n+V(q_n)\right]+\frac12\Tr\left(\Sigma^\top\Sigma\right) h\\
&=\E[H(p_{n},q_{n})]+\frac12\Tr\left(\Sigma^\top\Sigma\right) h.
\end{align*}
A recursion now completes the proof.
\end{proof}

The above result can also be seen as a longtime weak convergence result and provides a longtime stability result 
(in a certain sense) for the drift-preserving numerical scheme \eqref{dp}. Already in the case of a quadratic Hamiltonian, 
such trace formulas are not valid for classical numerical schemes for SDEs, as seen in \cite{smh04} for instance and in the numerical experiments provided in Section~\ref{sec:numerics}.

\section{Mean-square convergence analysis}\label{sect-ms}
In this section, we assume that $V^{\prime}$ is a globally Lipschitz continuous function and consider a compact time 
interval $[0,T]$. 
Then from the standard analysis of SDEs, see for instance \cite[Theorems~4.5.3~and~4.5.4]{Kloeden1992}, 
we known that for all $t,s\in[0,T]$ one has
\begin{align*}
&{\mathbb E}\left[|p(t)|^2+|q(t)|^2\right]\leq C,\\
&{\mathbb E}\left[|q(t)-q(s)|^2\right]\leq C|t-s|^2,\\
&{\mathbb E}\left[|p(t)-p(s)|^2\right]\leq C|t-s|,
\end{align*}
where the constant $C$ may depend on the coefficients of the SDE \eqref{prob} and its initial values.

We now state the result on the mean-square convergence of the drift-preserving integrator \eqref{dp}.
\begin{theorem}\label{ms-order}
Consider the stochastic Hamiltonian problem with additive noise~\eqref{prob} on a fixed time interval $[0,T]$ 
and the drift-preserving integrator~\eqref{dp}. 
Assume that the potential $V^{\prime}$ is globally Lipschitz continuous. 
Then, there exists $h^{*}>0$ such that for all $0<h\leq h^{*}$, the numerical scheme converges 
with order~$1$ in mean-square, i.\,e.,
\[
\left({\mathbb E}\left[|q(t_n)-q_n|^2\right]\right)^{1/2}+\left({\mathbb E}\left[|p(t_n)-p_n|^2\right]\right)^{1/2}\leq Ch,
\]
where the constant $C$ may depend on the initial data, the end time $T$, and the coefficients of \eqref{prob}, 
but is independent of $h$ and $n$ for $n=1, \ldots, N$. Here, $N$ is an integer 
such that $Nh=T$.  
\end{theorem}
\begin{proof}
We base the proof of this result on Milstein's fundamental convergence theorem, 
see \cite[Theorem~1.1]{Milstein2004} and first consider one step in the approximation 
of the numerical scheme~\eqref{dp}, starting from the point $(p, q)$ at time $t=0$ 
to $(\overline{p}, \overline{q})$ at time $t=h$:
\begin{equation}\label{os-dp}
\begin{split}
&\overline{\psi}=p+\Sigma W(h)-\frac{h}{2} \int_0^1 V^{\prime}(q+sh\overline{\psi})\,\text ds,\\
&\overline{q}=q+h\overline{\psi},\\
&\overline{p}=p+\Sigma W(h)-h \int_0^1 V^{\prime}(q+sh\overline{\psi})\,\text ds.
\end{split}
\end{equation}
Let us introduce a similar notation for the exact solution of \eqref{prob} starting from the point $(p, q)$ at time $t=0$ to $(p(h),q(h))$ at time $t=h$:
\begin{equation}\label{os-es}
\begin{split}
p(h)=&p-\int_0^h V^{\prime}(q(t))\,\text dt+\Sigma W(h),\\
q(h)=&q+\int_0^h p(t)\,\text dt=q+\int_0^h\Big[p-\int_0^t V^{\prime}(q(s))\,\text ds+\Sigma W(t)\Big]\, \text dt\\
=&q+h\Big[p-\frac{1}{h}\int_0^h\int_0^t V^{\prime}(q(s))\,\text ds\,\text dt+\frac{1}{h}\Sigma\int_0^h W(t)\,\text dt \Big]\\
=&q+h \psi(h).
\end{split}
\end{equation}

Finally, we define the local errors $e_{\text{loc}}^{q}=q(h)-\overline{q}$ and $e_{\text{loc}}^{p}=p(h)-\overline{p}$. 
In order to use Milstein's fundamental convergence theorem in our situation, one has to show that
\begin{align}
&\left|{\mathbb E}\left[e_{\text{loc}}^{q}\right]\right|=\bigo{h^2},\qquad \left({\mathbb E}\left[|e_{\text{loc}}^{q}|^2\right]\right)^{1/2}=\bigo{h^{3/2}}\label{es-q},\\
&\left|{\mathbb E}\left[e_{\text{loc}}^{p}\right]\right|=\bigo{h^2},\qquad \left({\mathbb E}\left[|e_{\text{loc}}^{p}|^2\right]\right)^{1/2}=\bigo{h^{3/2}}.\label{es-p}
\end{align}
It turns out that, due to the particular form of \eqref{prob}, 
we actually achieve better rates for the second term in \eqref{es-p}.
Using \eqref{os-dp} and \eqref{os-es}, we obtain
\begin{equation*}
\begin{split}
e_{\text{loc}}^{q}&=h(\psi(h)-\overline{\psi})\\
&=\Big(\Sigma \int_0^h W(t)\,\text dt-h\Sigma W(h)\Big)
-\Big( \int_0^h\int_0^t V^{\prime}(q(s))\,\text ds\,\text dt-\frac{h^2}{2}\int_0^1 V^{\prime}(q+sh\overline{\psi})\,\text ds \Big)\\
&=e_{\text{loc}}^{q,1}-e_{\text{loc}}^{q,2}.
\end{split}
\end{equation*}
Noting that 
\[
hW(h)=\int_0^h W(t)\,\text dt+\int_0^h t\,\text dW(t),
\]
one gets that
\[
e_{\text{loc}}^{q,1}=-\Sigma \int_0^h t\,\text dW(t).
\]
Hence by the property of It\^o integral, we have
\[
{\mathbb E}\left[e_{\text{loc}}^{q,1}\right]=0
\]
and
\[
{\mathbb E}\left[\left| e_{\text{loc}}^{q,1} \right|^2\right]={\mathbb E}\left[\left|\Sigma \int_0^h t\,\text dW(t) \right|^2\right]
=\sum_{i=1}^{m}{\mathbb E}\left[\left| \sum_{j=1}^{d} \sigma_{ij}\int_0^h t\,\text dW_{j}(t) \right|^2\right]
=\sum_{i=1}^{m}\sum_{j=1}^{d}(\sigma_{ij})^2\int_0^h t^2\,\text dt=\frac{h^3}{3}\left|\Sigma\right|_{F}^2, 
\]
where $\left|\cdot\right|_{F}$ denotes the Frobenius norm of a matrix. 
For the term $e_{\text{loc}}^{q,2}$, using H\"older's inequality and the linear growth condition on $V^{\prime}$, we obtain 
\begin{align*}
{\mathbb E}\left[\left| e_{\text{loc}}^{q,2}\right|^2\right]&\leq 2{\mathbb E}\left[\left| \int_0^h\int_0^t V^{\prime}(q(s))\,\text ds\,\text dt\right|^2\right]
+2{\mathbb E}\left[\left|\frac{h^2}{2}\int_0^1 V^{\prime}(q+sh\overline{\psi})\,\text ds \right|^2\right]\\
&\leq C_1h\int_0^h t\int_0^t {\mathbb E}\left[|V^{\prime}(q(s))|^2\right]\,\text ds\,\text dt
+C_2h^4\int_0^1 {\mathbb E}\left[|q+sh\overline{\psi}|^2\right]\,\text ds\\
&\leq C_1 h\int_0^h t\int_0^t (1+{\mathbb E}\left[|q(s)|^2\right])\,\text ds\,\text dt
+C_2h^4 \int_0^1 (1+{\mathbb E}\left[|q+sh\overline{\psi}|^2\right])\,\text ds\\
&\leq Ch^4.
\end{align*}
Here in the last step, we used the classical bounds from the beginning of this section, 
$$
{\mathbb E}\left[|q(s)|^2+|p(s)|^2\right]\leq C\left(1+{\mathbb E}\left[|q|^2\right]+{\mathbb E}\left[|p|^2\right]\right),
$$
and
\begin{equation}\label{psi}
{\mathbb E}\left[|\overline{\psi}|^2+|\overline{q}|^2+|\overline{p}|^2\right]\leq C\left(1+{\mathbb E}\left[|q|^2\right]+{\mathbb E}\left[|p|^2\right]\right).
\end{equation}
The last inequality comes from the fact that, from \eqref{os-dp}, we have 
\begin{align*}
{\mathbb E}\left[|\overline{\psi}|^2\right]&\leq C\left({\mathbb E}\left[|p|^2\right] 
+{\mathbb E}\left[|\Sigma W(h)|^2\right] +h^2 \int_0^1 {\mathbb E}\left[|V^{\prime}(q+sh\overline{\psi})|^2\right]\,
\text ds\right)\\
&\leq C_1{\mathbb E}\left[|p|^2\right]+C_2h+C_3h^2\left(1+{\mathbb E}\left[|q|^2\right]
+h^2 {\mathbb E}\left[|\overline{\psi}|^2\right]\right).
\end{align*}
This implies that there exists an $h^{*}>0$ such that for all $0<h\leq h^{*}$, inequality \eqref{psi} is verified.

Therefore, we obtain the following bounds $\left|{\mathbb E}\left[e_{\text{loc}}^{q,2}\right]\right|\leq \left({\mathbb E}\left[|e_{\text{loc}}^{q,2}|^2\right]\right)^{1/2}\leq Ch^2$ 
and hence we proved assertion \eqref{es-q}.

For the estimate of $e_{\text{loc}}^{p}$, we get from \eqref{os-dp} and \eqref{os-es} that
\begin{equation*}
e_{\text{loc}}^{p}=-\int_0^h V^{\prime}(q(t))\,\text dt+h\int_0^1 V^{\prime}(q+sh\overline{\psi})\,\text ds.
\end{equation*}
Then using that $V^{\prime}$ is globally Lipschitz continuous, we obtain
\begin{align*}
{\mathbb E}\left[|e_{\text{loc}}^{p}|^2\right]&={\mathbb E}\left[\left| h\int_0^1\left( V^{\prime}(q(hs))- V^{\prime}(q+sh\overline{\psi})\right)\,\text ds \right|^2\right]\\
&\leq h^2\int_0^1 {\mathbb E}\left[|V^{\prime}(q(hs))-V^{\prime}(q+sh\overline{\psi})|^2\right]\,\text ds\\
&\leq Ch^2\int_0^1\left( {\mathbb E}\left[|q(hs)-q|^2\right]+s^2h^2{\mathbb E}\left[|\overline{\psi}|^2\right] \right)\,\text ds\\
&\leq Ch^4.
\end{align*}
In the last step, we use the fact that 
\[
{\mathbb E}\left[|q(hs)-q|^2\right]={\mathbb E}\left[\left|\int_0^{hs}p(r)\,\text dr\right|^2\right]\leq hs \int_0^{hs}{\mathbb E}\left[|p(r)|^2\right]\,\text dr
\leq Ch^2.
\]
Therefore we obtain the bounds $\left|{\mathbb E}\left[e_{\text{loc}}^{p}\right]\right|\leq \left({\mathbb E}\left[|e_{\text{loc}}^{p}|^2\right]\right)^{1/2}\leq Ch^2$. 
An application of Milstein's fundamental convergence theorem completes the proof.
\end{proof}

\section{Weak convergence analysis}\label{sect-weak}

The weak error analysis for the Hamiltonian functional in the given situation is a direct consequence of the preceding sections. 
For the considered SDE, the main quantity of interest is given by a specific test function, i.\,e., we are interested in computing
\begin{equation*}
 \E\left[H(p(t),q(t))\right],
\end{equation*}
for $t>0$. Equation~\eqref{etrace} and Theorem~\ref{thmTrace} imply for all discrete times $t_n = nh$ that
\begin{equation*}
 \left| \E\left[H(p(t_n),q(t_n))\right]
            - \E\left[H(p_n,q_n)\right] \right|
    = 0.
\end{equation*}
For fixed $h>0$ let us further extend the discrete processes $(p_n, n=0,\ldots,N)$ and $(q_n, n=0,\ldots,N)$ to an adapted process for all $t>0$ by setting
\begin{equation*}
 p_h(t) = p_n,
 \qquad
 q_h(t) = q_n,
\end{equation*}
for $t \in [t_n,t_{n+1})$. Then $p_h$ and $q_h$ are piecewise continuous adapted processes which satisfy 
\begin{equation*}
  \left| \E\left[H(p(t),q(t))\right]
            - \E\left[H(p_h(t),q_h(t))\right] \right|
    = \frac{1}{2}\Tr\left(\Sigma^\top\Sigma\right)(t - t_n)
    \le \frac{1}{2} \Tr\left(\Sigma^\top\Sigma\right) h.
\end{equation*}
In conclusion using the trace formulas~\eqref{etrace} and~\eqref{ntrace} we have just shown the following corollary.

\begin{corollary}\label{cor:weakConv}
Assume that $V\in\mathcal C^1$. Consider the numerical discretization of the stochastic Hamiltonian system with additive noise~\eqref{prob} by 
the drift-preserving numerical scheme~\eqref{dp}, then the weak error in the Hamiltonian is bounded by
\begin{equation*}
  \left| \E\left[H(p(t),q(t))\right]
            - \E\left[H(p_h(t),q_h(t))\right] \right|
    \le \frac{1}{2} \Tr\left(\Sigma^\top\Sigma\right) h,
\end{equation*}
for all $t>0$. More specifically, for all $t = t_n$, with $n=0,\ldots,N$, the error is zero and the approximation scheme is preserving the quantity of interest.
\end{corollary}

Let us next consider the weak error for other test functions than~$H$ and introduce for convenience the space of square integrable random variables
\begin{equation*}
L^2(\Omega)
=
\left\{
v:\Omega \rightarrow \R, \, v \text{ strongly measurable}, \,
\| v \|_{L^2(\Omega)} < +\infty
\right\} ,
\end{equation*}
where
\begin{equation*}
\| v \|_{L^2(\Omega)}
=
\E [ |v|^2]^{1/2}.
\end{equation*}
In the following result on weak convergence, we obtain the same rate as for strong convergence in Section~\ref{sect-ms} and will see in the examples in Section~\ref{sec:numerics} that this seems optimal. This is caused by our combination of smoothness assumptions and the additive noise.

\begin{theorem}\label{weak-order}
Assume that $V'$ is globally Lipschitz continuous.
Then the drift preserving scheme~\eqref{dp} converges weakly with order~$1$ to the solution of~\eqref{prob} 
on any finite time interval~$[0,T]$. More specifically, for all differentiable test functions $f: \R^{2d} \rightarrow \R$ with $f'$ of polynomial growth there exist $C>0$ and $h^\ast > 0$ such that for all $0<h \le h^\ast$ and corresponding $n=1,\ldots,N$ with $hN = T$
\begin{equation*}
   \left| \E\left[f(p(t_n),q(t_n))\right]
            - \E\left[f(p_n,q_n)\right] \right|
    \le C h.
\end{equation*}
\end{theorem}

\begin{proof}
If $f$ is globally Lipschitz continuous, the claim follows directly from Theorem~\ref{ms-order}. 

Else let us use the abbreviations $X(t_n) = (p(t_n),q(t_n))$ and $X_n = (p_n,q_n)$. Then the mean value theorem and the Cauchy--Schwarz inequality imply that
\begin{align*}
 \left| \E\left[f(X(t_n))\right]
            - \E\left[f(X_n)\right] \right|
    & = \left| \E\left[\int_0^1 f'(X_n + s(X(t_n) - X_n)) \, \dd s \cdot (X(t_n) - X_n)\right] \right|\\
    & \le \int_0^1 \|f'(sX(t_n) + (1-s)X_n)\|_{L^2(\Omega)} \|X(t_n) - X_n\|_{L^2(\Omega)} \, \dd s.
\end{align*}
While the second term in the product is bounded by
\begin{equation*}
 \|X(t_n) - X_n\|_{L^2(\Omega)} \le C h
\end{equation*}
by Theorem~\ref{ms-order}, we use the polynomial growth assumption on~$f'$ for the first term to obtain
\begin{equation*}
 \int_0^1 \|f'(sX(t_n) + (1-s)X_n)\|_{L^2(\Omega)} \, \dd s
    \le C \left( 1 + 2^{(2m-1)/2} (2m+1)^{-1/2}
        \left( \E[|X(t_n)|^{2m}] + \E[|X_n|^{2m}] \right)^{1/2} \right).
\end{equation*}
The solution to~\eqref{prob} has bounded $2m$-th moment by Theorem~4.5.4 in~\cite{Kloeden1992}. The corresponding boundedness of the numerical solution follows by Lemma~2.2.2 in~\cite{Milstein2004} for $h$ sufficiently small depending on the polynomial growth constant on~$f'$. Choosing $h^\ast$ to be the minimum of this restriction and the one from Theorem~\ref{ms-order}, we conclude the proof.
\end{proof}

\begin{remark}
As we will see in Section~\ref{sec:numerics}, there exist combinations of equations and test functions for which the weak order of convergence of the drift-preserving scheme is in fact~$2$. More specifically, we observe this behaviour in simulations for the identity as test function. To understand this faster convergence, let us first consider the expectation of the stochastic 
Hamiltonian system \eqref{prob} 
\begin{align*}
\frac{\text d}{\text dt}\E[q(t)] &=\E[p(t)]\\
\frac{\text d}{\text dt}\E[p(t)] &=-\E[V'(q(t))]
\end{align*}
and apply the classical averaged vector field integrator (see e.g. \cite[Eq. (1.2)]{ch11}) 
to this ordinary differential equation. This time integrator is known to be of (deterministic) order of convergence $2$. 
We obtain the following numerical scheme
\begin{align*}
\E[q_{n+1}]&=\E[q_n]+\frac{h}{2}\left(\E[p_n]+\E[p_{n+1}]\right)\\
\E[p_{n+1}]&=\E[p_n]-h\int_0^1\E[V'((1-s)q_n+sq_{n+1})]\,\text ds.
\end{align*}
The above numerical scheme is nothing else than the expectation of the drift-preserving scheme \eqref{dp}
\begin{align*}
\E[\Psi_{n+1}]&=\E[p_n]-\frac{h}2\int_0^1\E[V'(q_n+sh\Psi_{n+1})]\,\text ds\\
\E[q_{n+1}]&=\E[q_n]+h\E[\Psi_{n+1}]\\
\E[p_{n+1}]&=\E[p_n]-h\int_0^1\E[V'(q_n+sh\Psi_{n+1})]\,\text ds=2\E[\Psi_{n+1}]-\E[p_n].
\end{align*}
We can thus conclude that the weak order of convergence in the first moment of the drift-preserving scheme is $2$. 
\end{remark}

In general, we will not be able to compute $\E\left[f(p_n,q_n)\right]$ analytically but have to approximate the expectation. This can for example be done by a standard Monte Carlo method. Following closely~\cite{L15}, we denote by 
\begin{equation*}
 E_M[Y] = M^{-1} \sum_{i=1}^M \hat{Y}^i
\end{equation*}
the Monte Carlo estimator of a random variable~$Y$ based on $M$ independent, identically distributed random variables~$\hat{Y}^i$. 
It is well-known that $E_M[Y]$ converges $\mathbb P$-almost surely to $\E[Y]$ for $M \rightarrow +\infty$, by the strong law of large numbers. Furthermore, it converges in mean square if $Y$ is square integrable and satisfies
\begin{equation*}
\| \E[Y] - E_M[Y] \|_{L^2(\Omega)}
  = M^{-1/2} \, {\mathsf{Var}}[Y]^{1/2}
  \leq M^{-1/2} \, \| Y \|_{L^2(\Omega)}.
\end{equation*}
By Corollary~\ref{cor:weakConv} for the Hamiltonian and the splitting of the error into the weak error in Theorem~\ref{weak-order} and the Monte Carlo error we therefore obtain the following lemma.

\begin{lemma}
 Assume that $V\in\mathcal C^1$. Consider the numerical discretization of the stochastic Hamiltonian system with additive noise~\eqref{prob} by 
the drift-preserving numerical scheme~\eqref{dp}, then the single level Monte Carlo error is bounded by
\begin{equation*}
  \left\| \E\left[H(p(t_n),q(t_n))\right]
            - E_M\left[H(p_n,q_n)\right] \right\|_{L^2(\Omega)}
    = M^{-1/2} \, {\mathsf{Var}}[H(p_n,q_n)]^{1/2}
    \le M^{-1/2} \, \|H(p_n,q_n)\|_{L^2(\Omega)},
\end{equation*}
for all $t_n = nh$. For any test function~$f$ under the assumptions of Theorem~\ref{weak-order}, this result generalizes 
on any finite time interval to
\begin{align*}
  \left\| \E\left[f(p(t_n),q(t_n))\right]
            - E_M\left[f(p_n,q_n)\right] \right\|_{L^2(\Omega)}
    & = C h + M^{-1/2} \, {\mathsf{Var}}[f(p(t_n),q(t_n))]^{1/2}\\
    &\le C h + M^{-1/2} \, \|f(p(t_n),q(t_n))\|_{L^2(\Omega)}
\end{align*}
with computational work for balanced error contributions
\begin{equation*}
 \cW = h^{-1} \cdot M^2
    = h^{-3}.
\end{equation*}
\end{lemma}
For completeness of presentation of the numerical analysis of the drift-preserving scheme \eqref{dp}, 
although not necessarily relevant for Hamiltonian SDEs, let us look at the savings in computational costs 
of the multilevel Monte Carlo estimator~\cite{H01, G08}. 
Therefore we assume that $(Y_\ell, \ell \in \N_0)$ is a sequence of approximations of~$Y$. For given $L \in \N_0$, it holds that
\begin{equation*}
 Y_L = Y_0 + \sum_{\ell = 1}^L (Y_\ell - Y_{\ell - 1})
\end{equation*}
and due to the linearity of the expectation that
\begin{equation*}
  \E[Y_L] = \E[Y_0] + \sum_{\ell = 1}^L \E[Y_\ell - Y_{\ell - 1}].
\end{equation*}
A possible way to approximate $\E[Y_L]$ is to approximate $\E[Y_\ell - Y_{\ell - 1}]$ with the corresponding Monte Carlo estimator $E_{M_\ell}[Y_\ell - Y_{\ell - 1}]$ with a number of independent samples~$M_\ell$ depending on the level~$\ell$. We set
\begin{equation*}
 E^L[Y_L]
    = E_{M_0}[Y_0] + \sum_{\ell=1}^L E_{M_\ell}[Y_\ell - Y_{\ell - 1}]
\end{equation*}
and call $E^L[Y_L]$ the \emph{multilevel Monte Carlo estimator of $\E[Y_L]$}. 

We consider for the remainder of this section for convenience that the approximation scheme is based on a sequence of equidistant nested time discretizations $\tau = (\tau^\ell, \ell \in \N_0)$ given by
\begin{equation*}
 \tau^\ell
    = (t_n^\ell = T \cdot n \cdot 2^{-\ell}, n=0,\ldots,2^\ell)
\end{equation*}
with sequence of step sizes $(h_\ell = T\cdot 2^{-\ell}, \ell \in \N_0)$. 

Let us denote by $((p^\ell,q^\ell), \ell \in \N_0)$ the sequence of approximation schemes with respect to the sequence of time discretizations~$\tau$. We observe that Theorem~\ref{weak-order} implies in the setting of \cite[Theorem~1]{L15} that $a_\ell = 2^{-\ell}$, $\eta = 2^{-1}$ by Theorem~\ref{ms-order}, and $\kappa = 1$. Plugging in these values we obtain the following corollary.

\begin{corollary}
Consider the numerical discretizations~\eqref{dp} of the stochastic Hamiltonian system~\eqref{prob} satisfying the assumptions in Theorem~\ref{ms-order} and Theorem~\ref{weak-order}. Then for any differentiable test function~$f$ with derivative of polynomial growth, the multilevel Monte Carlo estimator on level~$L>0$ satisfies for any $\epsilon > 0$ at the final time~$T$
\begin{equation*}
 \|\E[f(p(T),q(T))] - E^L[f(p_{2^L}^L, q_{2^L}^L)]\|_{L^2(\Omega)}
  \le (C_1 + C_3 + C_2 \zeta(1 + \epsilon)) \, h_L
\end{equation*}
with sample sizes given by
$M_\ell = \lceil 2^{2(L - \ell/2)} \ell^{2(1+\epsilon)}\rceil$, $\ell = 1,\ldots,L$, $\epsilon > 0$, and $M_0 = \lceil 2^{2L} \rceil$, 
where $\lceil \cdot \rceil$ denotes the rounding to the next larger integer and
$\zeta$ the Riemann zeta function. 
The resulting computational work is bounded by
\begin{equation*}
 \cW_L
    = \bigo{h_L^{-2}L^{3+2\epsilon}}.
\end{equation*}
\end{corollary}

In conclusion we have seen that our drift-preserving scheme converges in general weakly with order~$1$, i.\,e., 
the same order as in mean square in Section~\ref{sect-ms}. 
Approximating quantities of interest with a standard Monte Carlo error with accuracy~$h$ requires computational work~$h^{-3}$. 
This can be reduced to essentially~$h^{-2}$ when a multilevel Monte Carlo estimator is applied instead.

\section{Numerical experiments}\label{sec:numerics}
This section presents various numerical experiments in order to illustrate 
the main properties of the drift-preserving scheme \eqref{dp}, denoted by \textsc{DP} below. 
In some numerical experiments, we will compare this numerical scheme with classical ones for SDEs 
such as the Euler--Maruyama scheme (denoted by \textsc{EM}) and 
the backward Euler--Maruyama scheme (denoted by \textsc{BEM}). 

\subsection{The linear stochastic oscillator}
We first consider the SDE \eqref{prob} with the following Hamiltonian
$$
H(p,q)=\frac12p^2+\frac12q^2
$$
and with $\Sigma=1$ and $W$ scalars. We take the initial values $(p_0,q_0)=(0,1)$. 

In this situation the integral in the drift-preserving scheme \eqref{dp} can be computed exactly, 
resulting in the following time integrator
\begin{equation*}
\begin{split}
\Psi_{n+1}&=\left(p_n+\Delta W_n-\frac{h}2q_n\right)\left(1+\frac{h^2}4\right)^{-1},\\
q_{n+1}&=q_n+h\Psi_{n+1},\\
p_{n+1}&=p_n+\Delta W_n-h\left(q_n+\frac{h}2\Psi_{n+1}\right).
\end{split}
\end{equation*}

Using the stepsize $h=5/2^4$, resp. $h=100/2^8$, and the time interval $[0,5]$, resp. $[0,150]$, we compute the expected values of the energy $H(p,q)$ along the numerical solutions. 
For this problem, we also use the stochastic trigonometric method from \cite{c12}, denoted by \textsc{STM}, 
which is know to preserve the trace formula for the energy for this problem. 
For comparison, the following splitting strategies are also used
\begin{itemize}
\item composition of the (deterministic) symplectic Euler scheme for the Hamiltonian part 
with an analytical integration of the Brownian motion (this scheme is denoted by \textsc{SYMP});
\item a splitting based on the decomposition  
$\text dq(t) = p(t)\,\text dt, \text dp(t) = \Sigma\, \text dW(t)$ 
and $\text dq(t) = 0\,\text dt, \text dp(t) = -V'(q(t))\,\text dt$ (this time integrator is denoted by \textsc{SPLIT}).
\end{itemize}

The expected values are approximated by computing averages over $M=10^6$ samples.
The results are presented in Figure~\ref{fig:traceSO}, where one can clearly observe the excellent behaviour of the drift-preserving scheme as stated 
in Theorem~\ref{thmTrace}. Observe that it can be shown that the expected value of the Hamiltonian 
along the Euler--Maruyama scheme drifts exponentially with time. Furthermore, 
the growth rate of this quantity along the backward Euler--Maruyama scheme is slower 
than the growth rate of the exact solution to the considered SDE, 
see \cite{smh04} for details. These growth rates are qualitatively different from the linear growth rate 
in the expected value of the Hamiltonian of the exact solution \eqref{etrace}, 
of the \textsc{STM} from \cite{c12}, 
and of the drift-preserving scheme \eqref{ntrace}. 
Although not having the correct growth rates, 
the splitting schemes behave much better than the classical Euler--Maruyama schemes. 
Further splitting strategies are under investigation in \cite{cv20}.

\begin{figure}[h]
\centering
\includegraphics*[height=5cm,keepaspectratio]{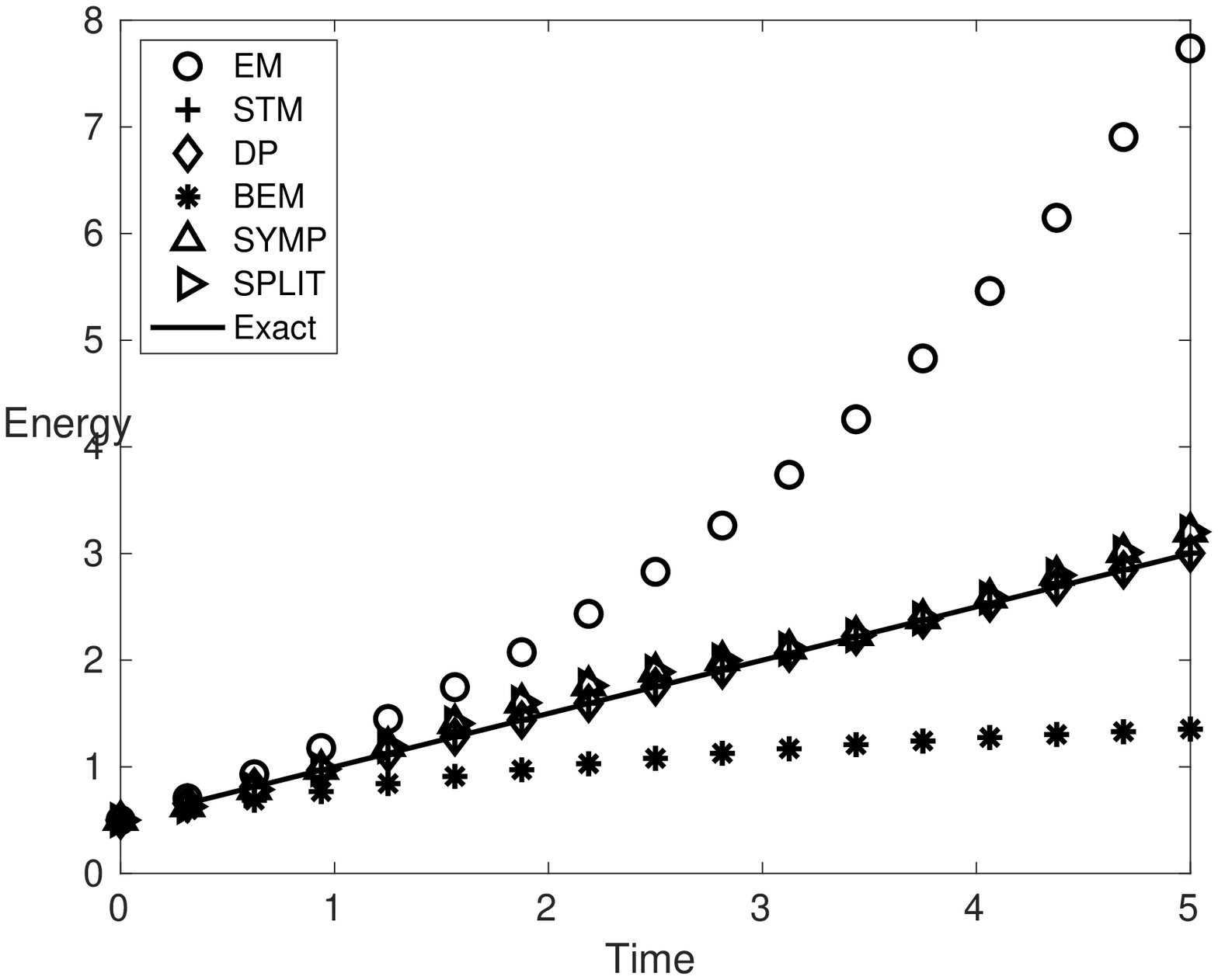}
\includegraphics*[height=5cm,keepaspectratio]{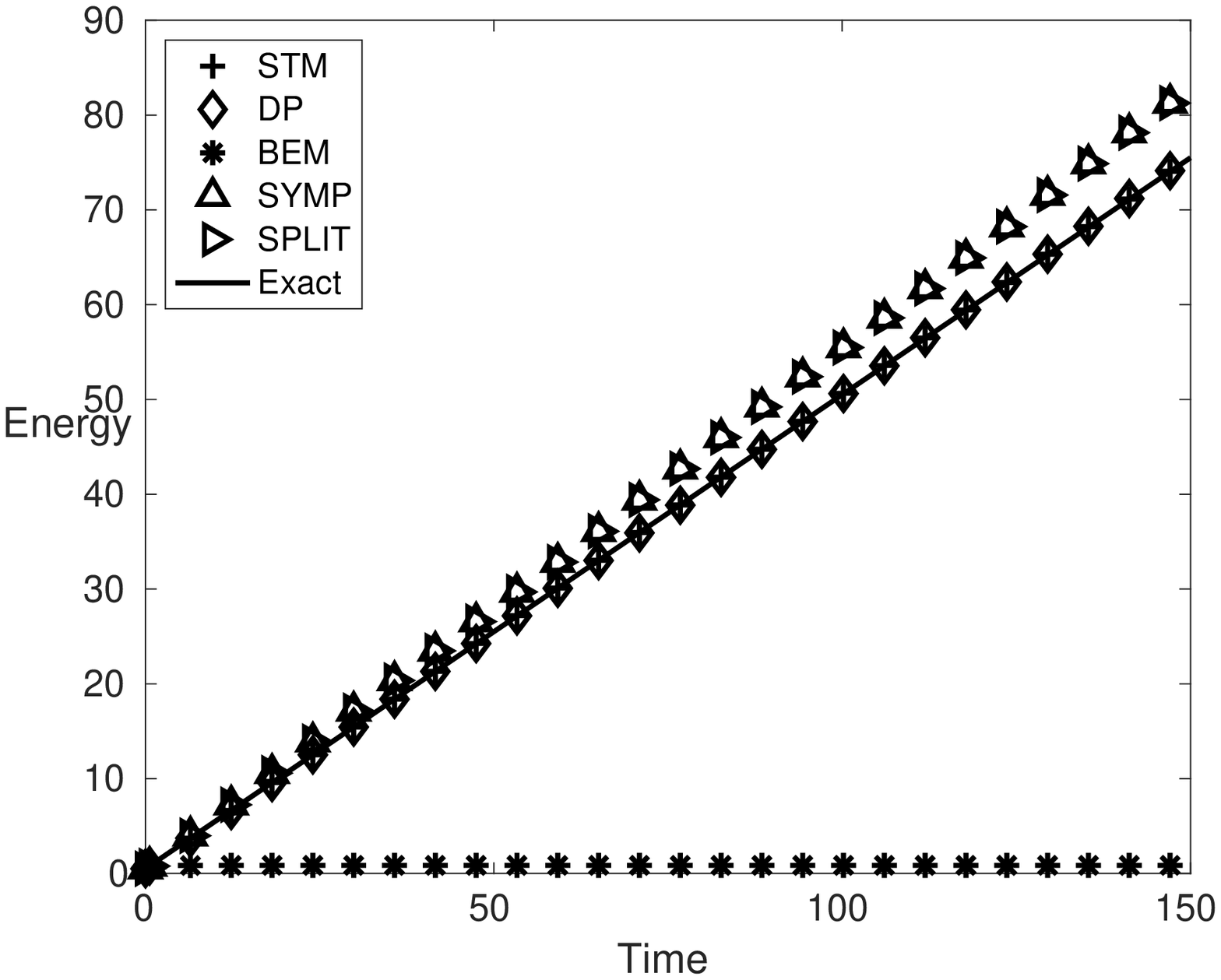}
\caption{Numerical trace formulas for the linear stochastic oscillator on $[0,5]$ (left) and $[0,150]$ (right).}
\label{fig:traceSO}
\end{figure}

We next illustrate numerically the strong convergence rate of the drift-preserving scheme stated in Theorem~\ref{ms-order}. 
Using the same parameters as above, we discretize the SDE on the interval $[0,1]$ 
using step sizes ranging from $2^{-5}$ to $2^{-10}$. The loglog plots of the errors are presented in Figure~\ref{fig:msSO}, 
where mean-square convergence of order $1$ for the proposed integrator is observed. The reference solution is computed 
with the stochastic trigonometric method using $h_{\text{ref}}=2^{-12}$. The expected values are approximated 
by computing averages over $M=10^5$ samples.

\begin{figure}[h]
\centering
\includegraphics*[height=5cm,keepaspectratio]{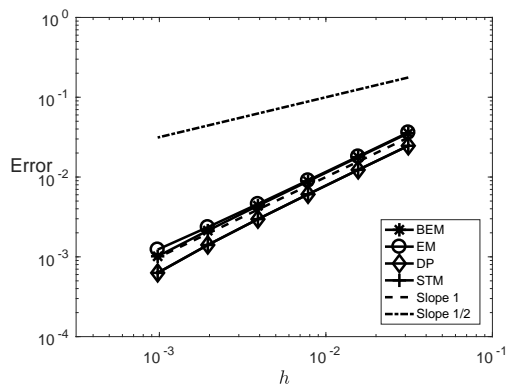}
\caption{Strong rates of convergence for the drift-preserving scheme (DP), the backward Euler--Maruyama scheme (BEM), the Euler--Maruyama scheme (EM), 
and the stochastic trigonometric method (STM) when applied 
to the linear stochastic oscillator.}
\label{fig:msSO}
\end{figure}

To conclude this subsection, we numerically illustrate the weak rates of convergence of the above time integrators. 
In order to avoid Monte Carlo approximations, we focus on weak errors in the first and second moments only, where 
all the expectations can be computed exactly. We use the same parameters as above except 
for $\Sigma=0.1$, $T=1$, and step sizes ranging from $2^{-4}$ to $2^{-16}$. 
The results are presented in Figure~\ref{fig:weakSO}, where one can observe weak order $2$ in the first moment and 
weak order $1$ in the second moment for the drift-preserving scheme. This is in accordance with the results from the 
preceding section.

\begin{figure}[h]
\centering
\begin{subfigure}[b]{0.55\textwidth}
   \includegraphics*[height=3.1cm,keepaspectratio]{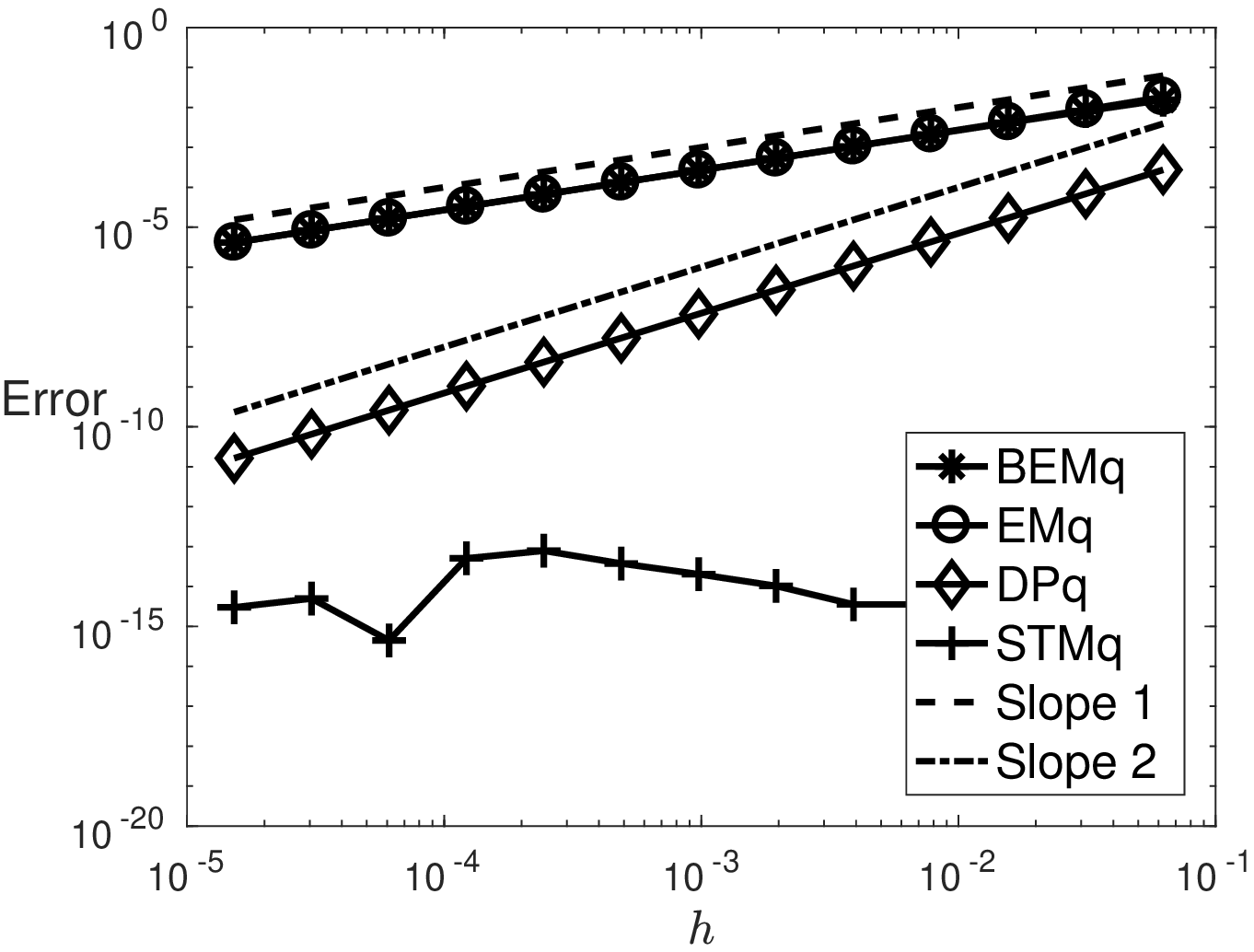}
   \includegraphics*[height=3.1cm,keepaspectratio]{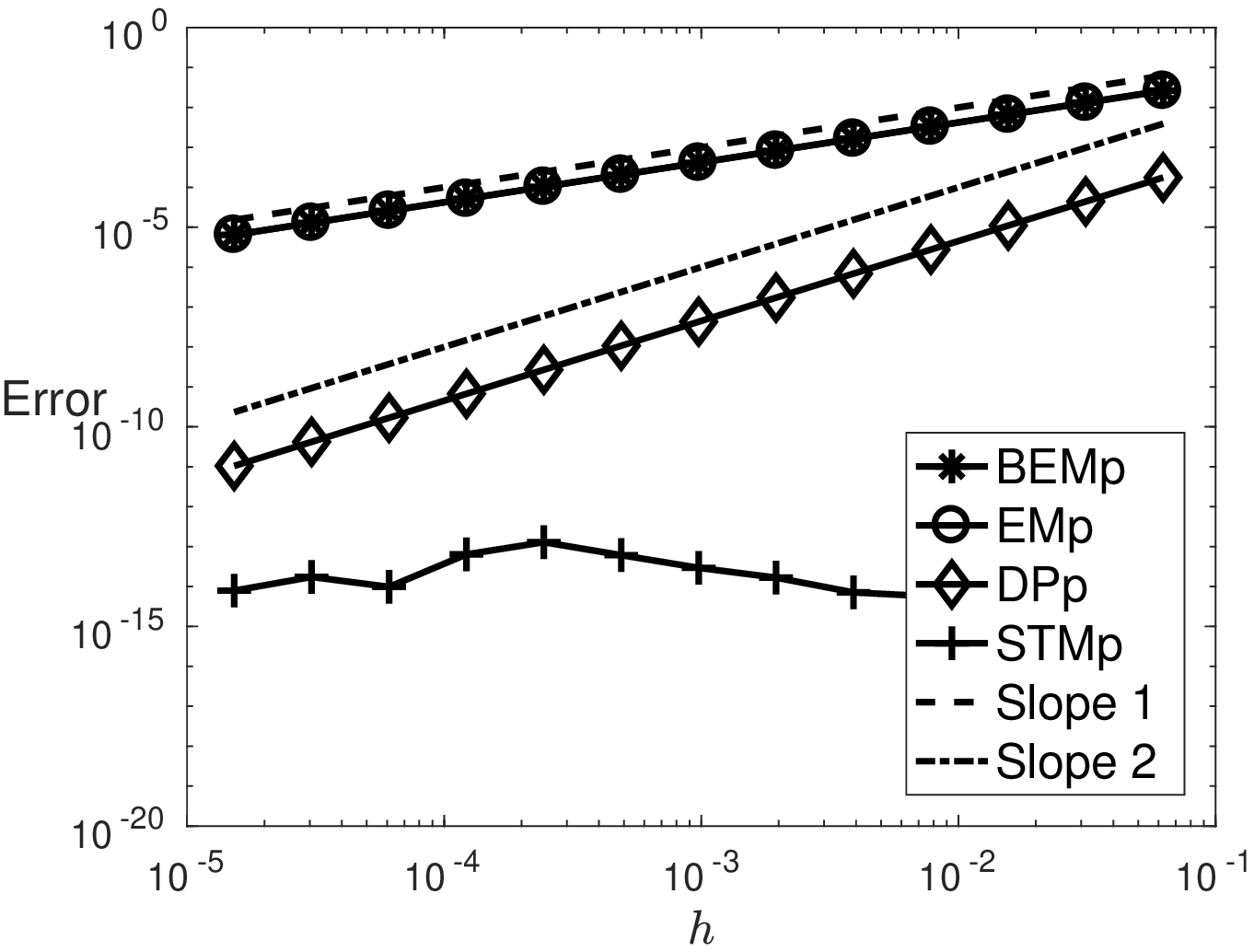}
   \caption{Errors in the first moment of $q$ (left) and $p$ (right).}
   \label{fig:w1} 
\end{subfigure}
\begin{subfigure}[b]{0.55\textwidth}
   \includegraphics*[height=3.1cm,keepaspectratio]{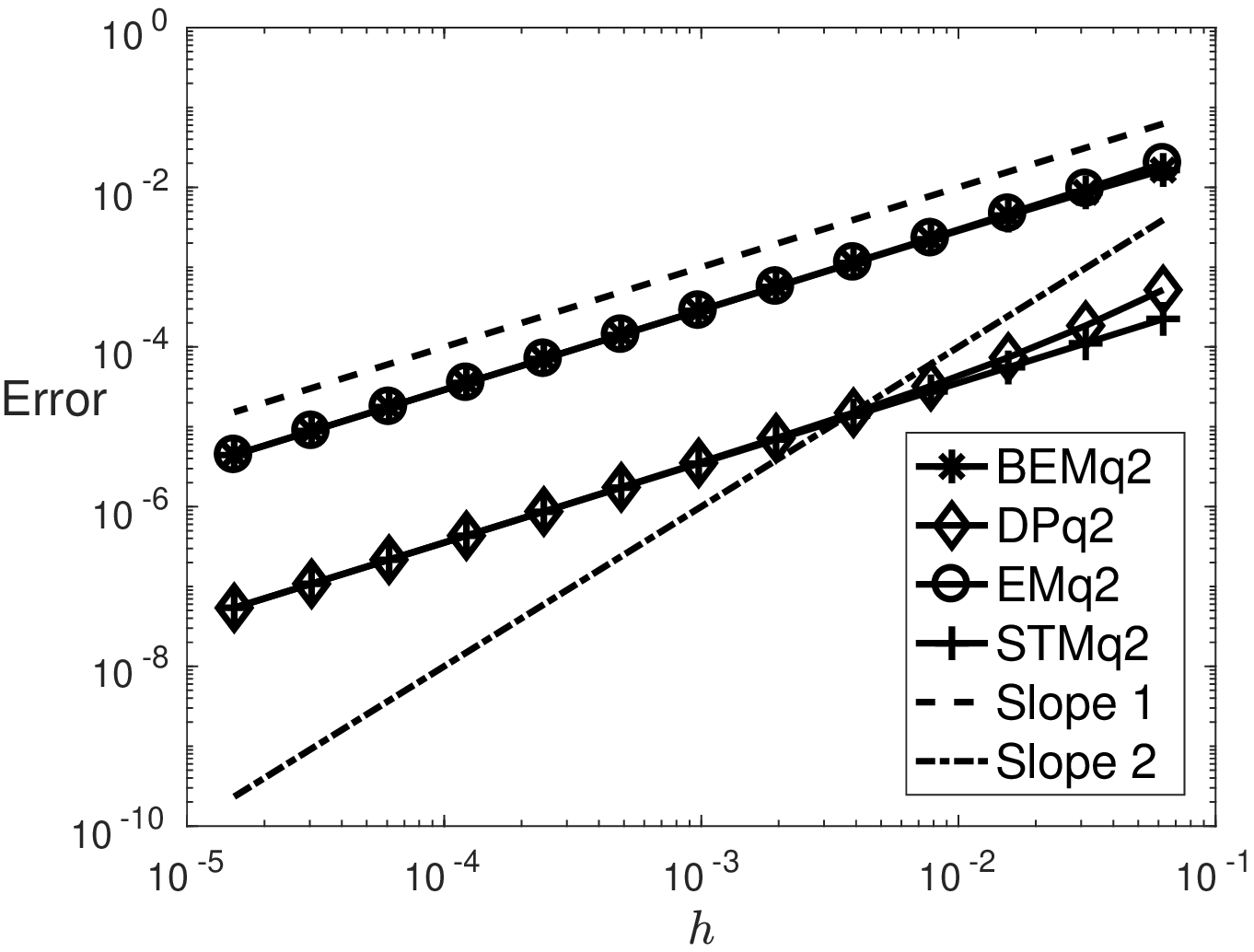}
   \includegraphics*[height=3.1cm,keepaspectratio]{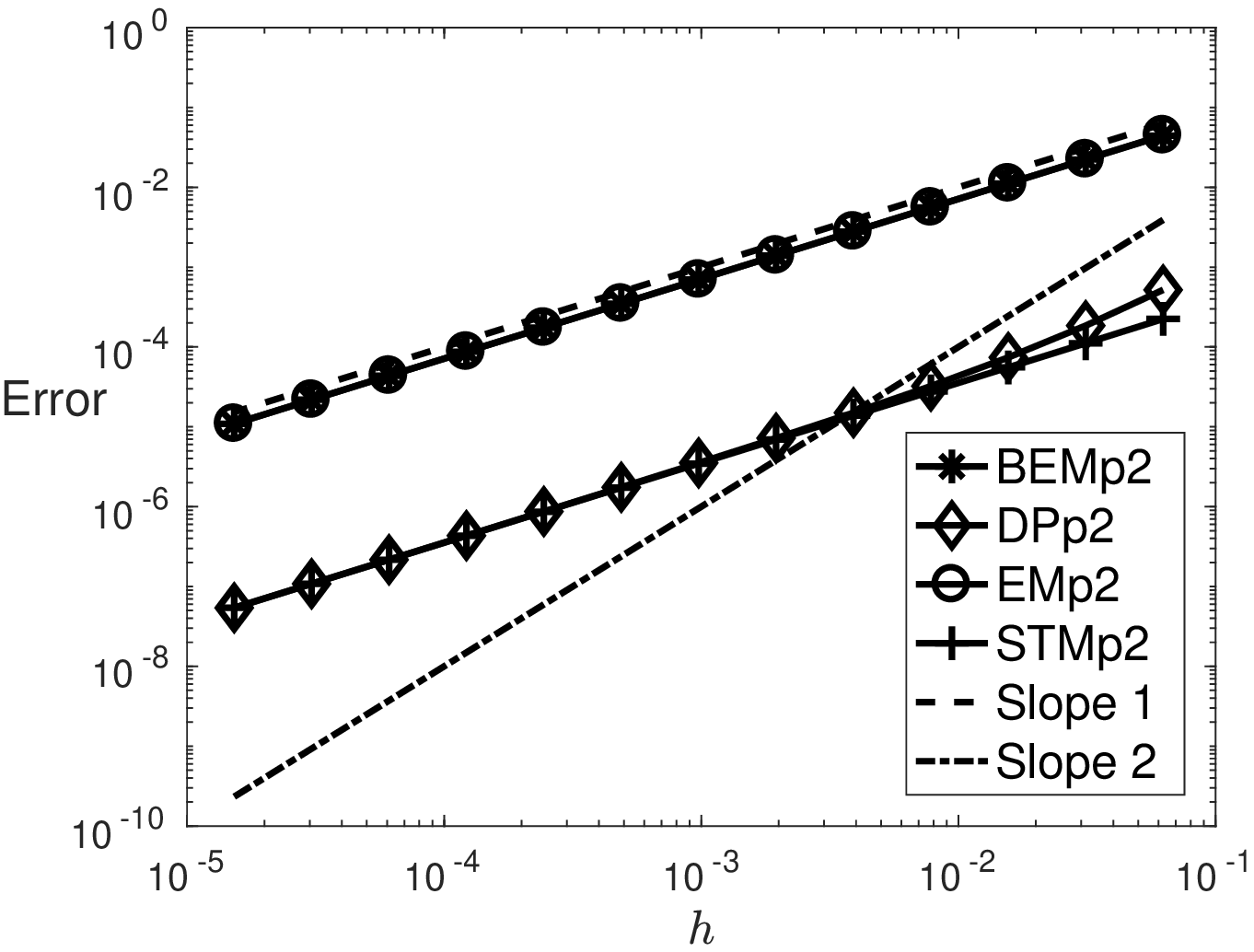}
   \caption{Errors in the second moment of $q$ (left) and $p$ (right).}
   \label{fig:w2} 
\end{subfigure}
\caption{Weak rates of convergence for the drift-preserving scheme (DP), 
the backward Euler--Maruyama scheme (BEM), the Euler--Maruyama scheme (EM), 
and the stochastic trigonometric method (STM) when applied 
to the linear stochastic oscillator.}
\label{fig:weakSO}
\end{figure}

\subsection{The stochastic mathematical pendulum}
Let us next consider the SDE~\eqref{prob} with the Hamiltonian
$$
H(p,q)=\frac12p^2-\cos(q)
$$
and with $\Sigma=0.25$ and $W$ scalars. We take the initial values $(p_0,q_0)=(1,\sqrt{2})$. 

For this problem, the proposed time integrator reads
\begin{align*}
\Psi_{n+1}&=p_n+\Sigma\Delta W_n-\frac{h}2\left(\frac{\cos(q_n)-\cos(q_n+h\Psi_{n+1})}{h\Psi_{n+1}}\right),\\
q_{n+1}&=q_n+h\Psi_{n+1},\\
p_{n+1}&=p_n+\Sigma\Delta W_n-h\left(\frac{\cos(q_n)-\cos(q_n+h\Psi_{n+1})}{h\Psi_{n+1}}\right).
\end{align*}

Using the stepsize $h=5/2^8$, resp. $h=10/2^{10}$, and the time interval $[0,5]$, resp. $[0,10]$, 
we compute the expected values of the energy $H(p,q)$ along the numerical solutions. 
Newton's iterations are used to solve the nonlinear systems in the drift-preserving scheme \eqref{dp} as well as in the BEM scheme. 
The expected values are approximated by computing averages over $M=10^6$ samples, resp. $M=10^5$ samples. 
The results are presented in Figure~\ref{fig:tracePend}, where one can clearly observe the perfect behaviour of the drift-preserving scheme as stated 
in Theorem~\ref{thmTrace} as well as the excellent behaviour of the splitting strategies. 

\begin{figure}[h]
\centering
\includegraphics*[height=5cm,keepaspectratio]{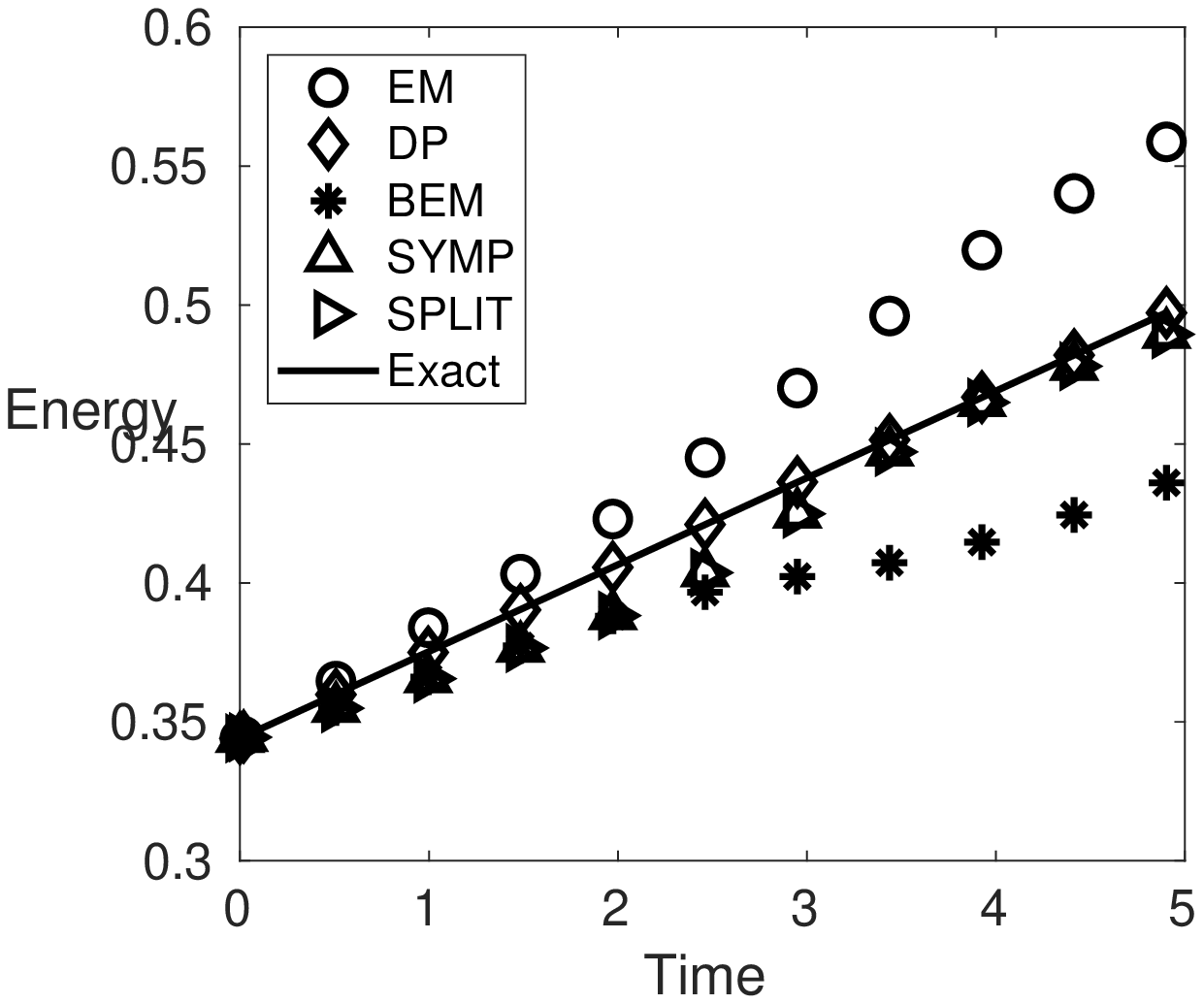}
\includegraphics*[height=5cm,keepaspectratio]{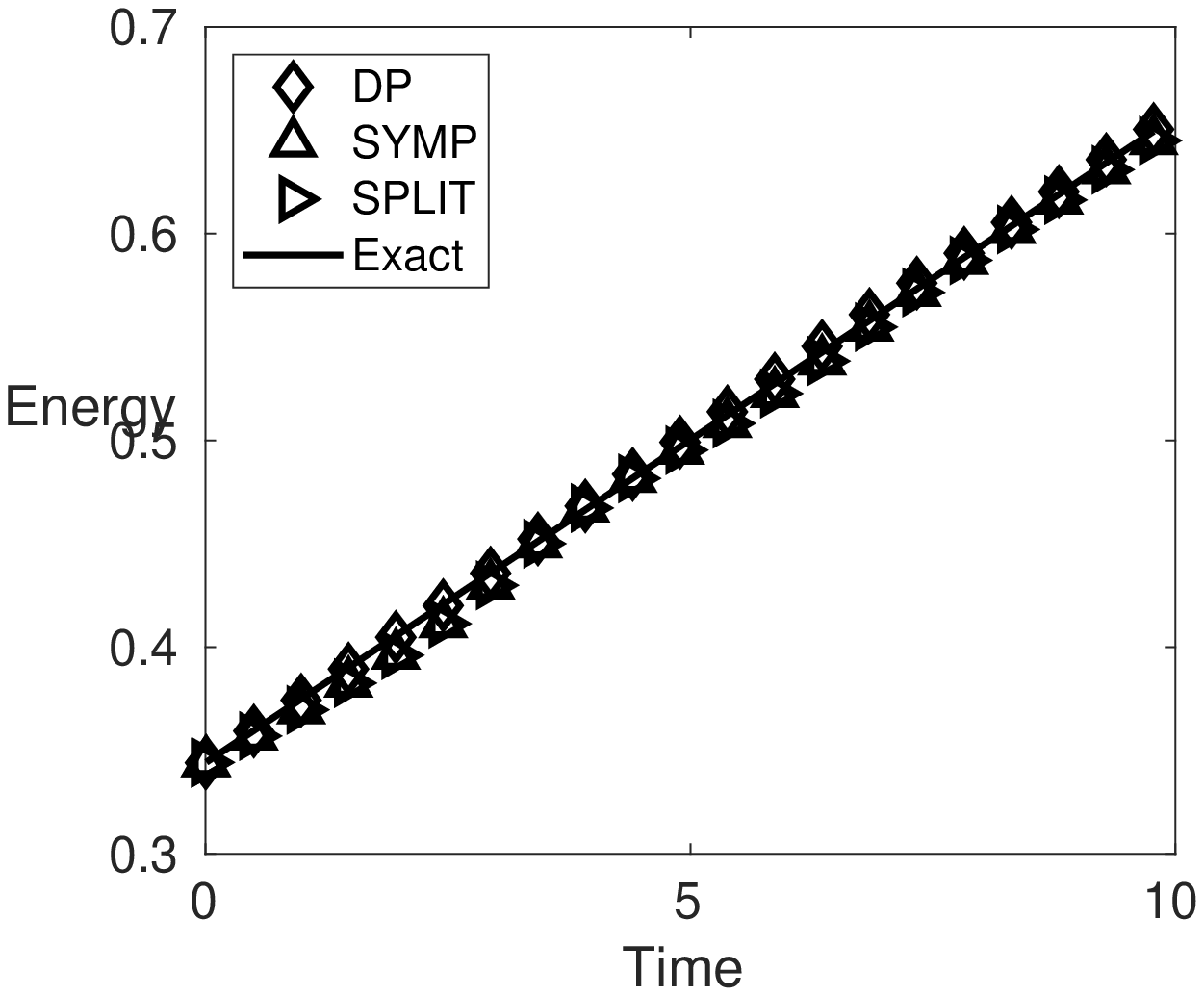}
\caption{Numerical trace formulas for the stochastic mathematical pendulum on $[0,5]$ (left) and $[0,10]$ (right).}
\label{fig:tracePend}
\end{figure}

We also illustrate numerically the strong convergence rate of the drift-preserving scheme stated in Theorem~\ref{ms-order}. 
For this, we discretize the stochastic mathematical pendulum with $\Sigma=0.1$ on the interval $[0,0.5]$ 
using step sizes ranging from $2^{-5}$ to $2^{-10}$. The loglog plots of the errors are presented in Figure~\ref{fig:msPend}, 
where mean-square convergence of order $1$ for the proposed integrator is observed. The reference solution is computed 
with the drift-preserving scheme using $h_{\text{ref}}=2^{-12}$. The expected values are approximated 
by computing averages over $M=10^5$ samples.

\begin{figure}[h]
\centering
\includegraphics*[height=5cm,keepaspectratio]{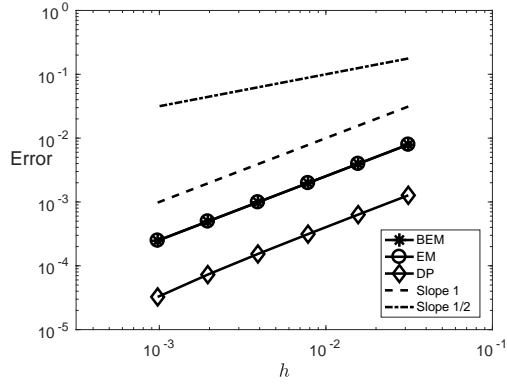}
\caption{Strong rates of convergence for the drift-preserving scheme (DP), the backward Euler--Maruyama scheme (BEM), and the Euler--Maruyama scheme (EM), 
when applied to the stochastic mathematical pendulum.}
\label{fig:msPend}
\end{figure}

We conclude this subsection by reporting numerical experiments illustrating the weak rates of convergence of the above time integrators. 
In order to reduce the Monte Carlo error and thus produce nice plots, we had to multiply the nonlinearity with a small coefficient of $0.2$ 
and considered the interval $[0,1]$. The other parameters are the same as above. The step sizes range from $2^{-1}$ to $2^{-6}$. 
The expected values are approximated by computing averages over $M=10^8$ samples.
The results are presented in Figure~\ref{fig:weakPend}, where one can observe weak order $2$, resp. $1$, 
of convergence for the drift-preserving scheme in the first, resp. second, moment in the variable $q$. 
This confirms the theoretical results from the previous section.

\begin{figure}[h]
\centering
\includegraphics*[height=5.cm,keepaspectratio]{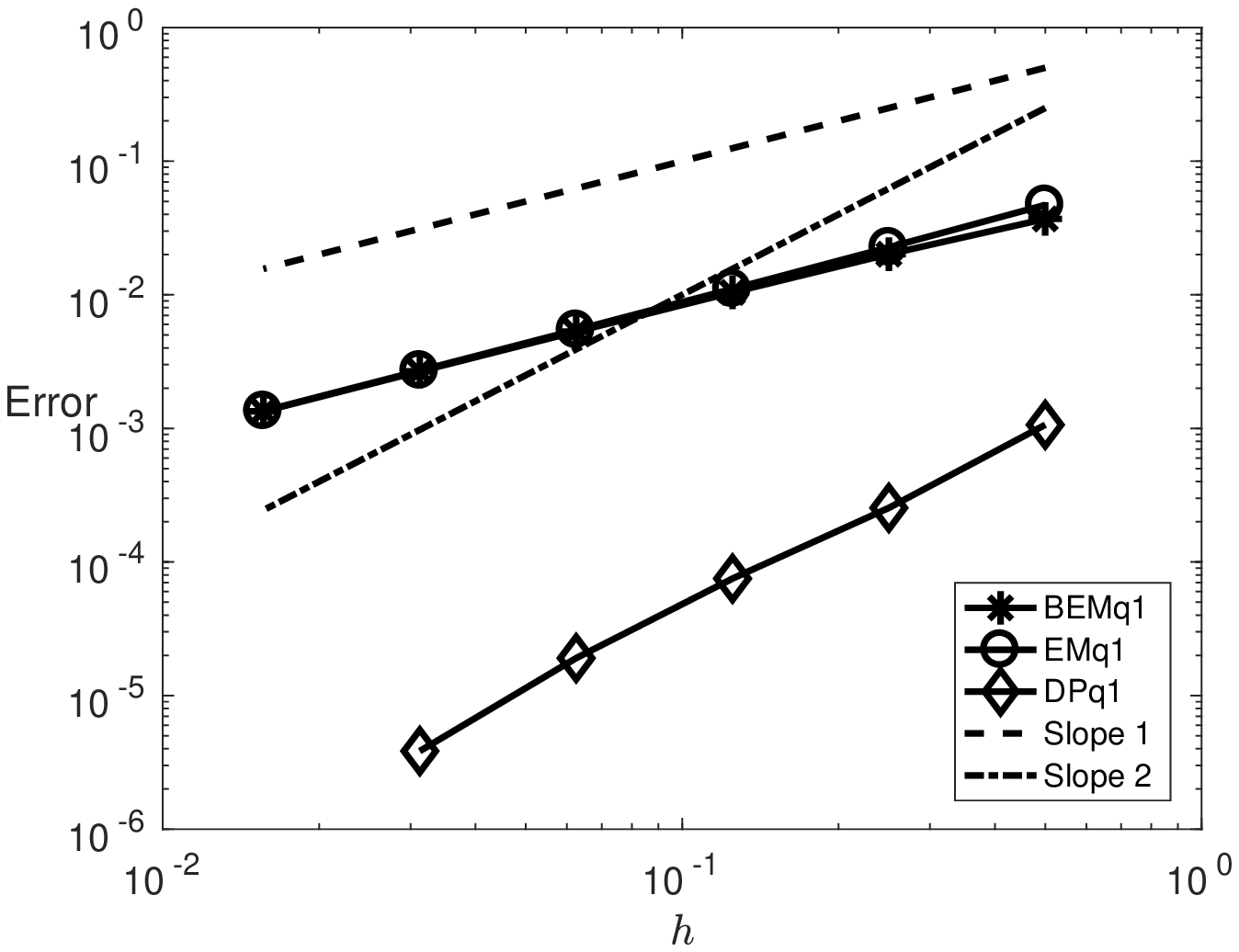}
\includegraphics*[height=5.cm,keepaspectratio]{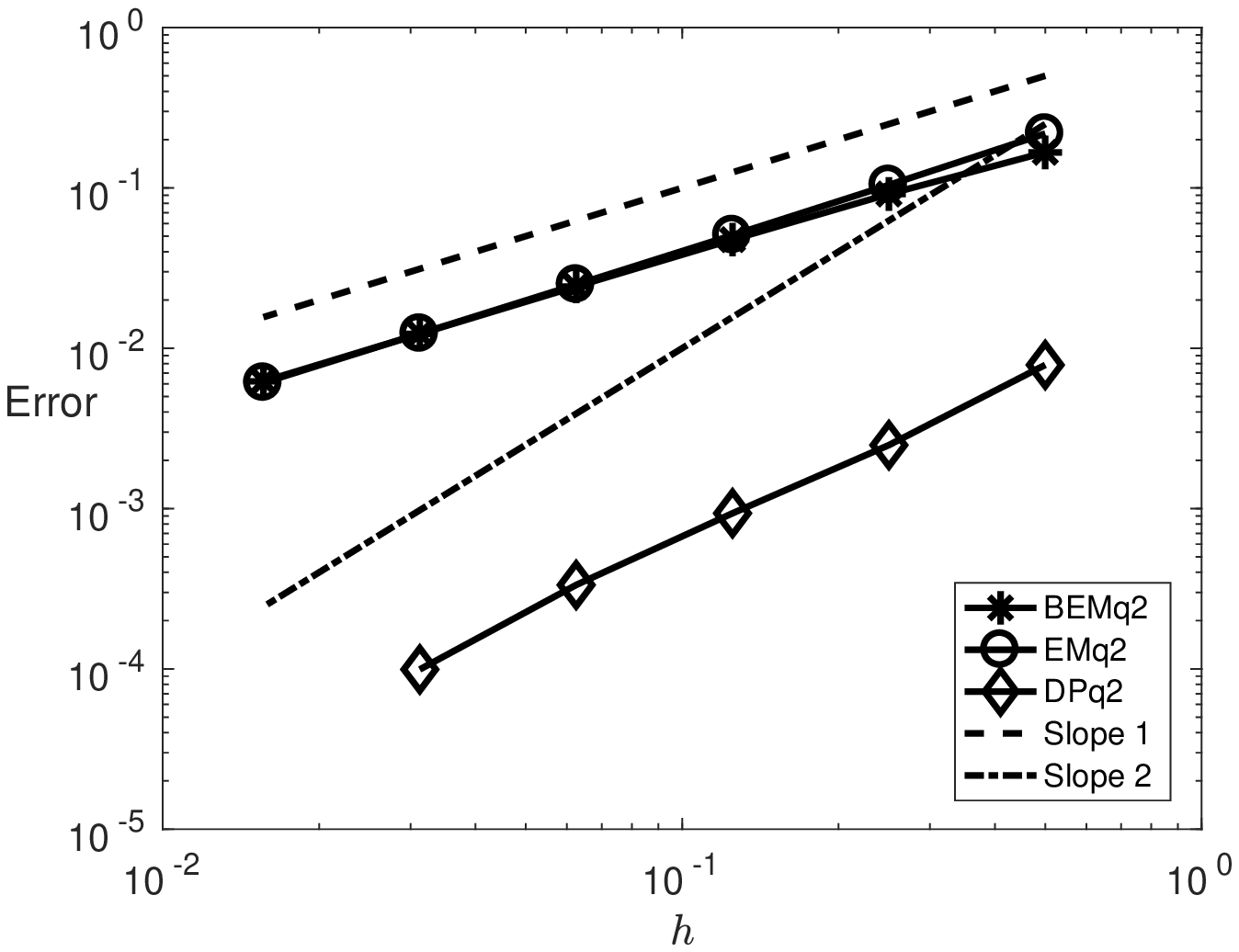}
\caption{Weak rates of convergence for the drift-preserving scheme (DP), 
the backward Euler--Maruyama scheme (BEM), and the Euler--Maruyama scheme (EM) 
when applied to the stochastic mathematical pendulum. Errors in the first moment of $q$ (left) 
and in the second moment of $q$ (right).}
\label{fig:weakPend}
\end{figure}

\subsection{Double well potential}
We consider the Hamiltonian with double well potential from \cite{bb14}. The SDE~\eqref{prob} is thus given by the Hamiltonian
$$
H(p,q)=\frac12p^2+\frac14q^4-\frac12q^2
$$
and with $\Sigma=0.5$ and $W$ scalars. We take the initial values $(p_0,q_0)=(\sqrt{2},\sqrt{2})$. 

When applied to the Hamiltonian with double well potential, the time integrator \eqref{dp} takes the form
\begin{align*}
\Psi_{n+1}&=p_n+\Sigma\Delta W_n-\frac{h}2\left( q_n^3-q_n-\frac{h}2\Psi_{n+1}+\frac{3h}{2}q_n^2\Psi_{n+1}+h^2q_n\Psi_{n+1}^2+\frac{h^3}4\Psi_{n+1}^3 \right),\\
q_{n+1}&=q_n+h\Psi_{n+1},\\
p_{n+1}&=p_n+\Sigma\Delta W_n-h\left( q_n^3-q_n-\frac{h}2\Psi_{n+1}+\frac{3h}{2}q_n^2\Psi_{n+1}+h^2q_n\Psi_{n+1}^2+\frac{h^3}4\Psi_{n+1}^3 \right).
\end{align*}

Using $2^{16}$ stepsizes of the drift-preserving scheme \eqref{dp} on the time interval $[0,50]$ 
(using fixed-point iterations for solving the implicit systems), and approximating 
the expected value with $M=10^5$ samples, we obtain 
the result displayed in Figure~\ref{fig:traceDouble}. This again numerically confirms  
the long time behaviour of the drift-preserving scheme stated in Theorem~\ref{thmTrace} 
and shows its superiority compared to the numerical schemes from \cite{bb14}, see the numerical results in \cite[Table~1]{bb14}.

\begin{figure}[h]
\centering
\includegraphics*[height=5cm,keepaspectratio]{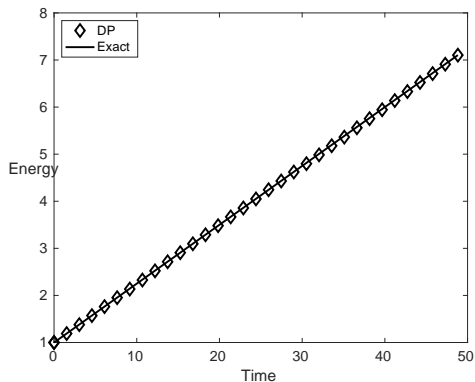}
\caption{Numerical trace formula for the stochastic Hamiltonian with double well potential on $[0,50]$.}
\label{fig:traceDouble}
\end{figure}

\subsection{H\'enon--Heiles problem with two additive noises}
Finally, we consider the H\'enon--Heiles problem with two additive noises from \cite{bb14,MR3873562}. This SDE is given 
by the Hamiltonian
$$
H(p,q)=\frac12\left(p_1^2+p_2^2\right)+\frac12\left(q_1^2+q_2^2\right)+\alpha\left(q_1q_2^2-\frac13q_1^3\right),
$$
and with $\Sigma=\text{diag}(\sigma_1,\sigma_2)$ and $W=( W^1, W^2)^\top$ 
in \eqref{prob}. We take the following parameters $\sigma_1=\sigma_2=0.2$, $\alpha=1/16$, 
and initial values $p_{0}=(1,1)$ and $q_0=(\sqrt{3},1)$.

For this system of SDEs, the drift-preserving scheme reads
\begin{align*}
&\begin{pmatrix}
\Psi_{1,n+1}\\
\Psi_{2,n+1}
\end{pmatrix}\\
&=
\begin{pmatrix}
p_{1,n}\\
p_{2,n}
\end{pmatrix}
+
\begin{pmatrix}
\sigma_1\Delta W^1_{n}\\
\sigma_2\Delta W^2_{n}
\end{pmatrix}
\\
&\quad-\frac{h}2
\begin{pmatrix}
q_{1,n}+\frac{h}2\Psi_{1,n+1}+\alpha\left( q_{2,n}^2+h\Psi_{2,n+1}q_{2,n}+\frac{h^2}3\Psi_{2,n+1}^2 \right)-\alpha\left( q_{1,n}^2+h\Psi_{1,n+1}q_{1,n}+\frac{h^2}3\Psi_{2,n+1}^2 \right)\\
q_{2,n}+\frac{h}2\Psi_{2,n+1}+2\alpha\left( q_{1,n}q_{2,n}+\frac{h}2q_{1,n}\Psi_{2,n+1}+\frac{h}2q_{2,n}\Psi_{1,n+1}+\frac{h^2}3\Psi_{1,n+1}\Psi_{2,n+1}  \right)
\end{pmatrix},\\
&q_{n+1}=q_n+h\Psi_{n+1}\\
&\begin{pmatrix}
p_{1,n+1}\\
p_{2,n+1}
\end{pmatrix}\\
&=
\begin{pmatrix}
p_{1,n}\\
p_{2,n}
\end{pmatrix}
+
\begin{pmatrix}
\sigma_1\Delta W^1_{n}\\
\sigma_2\Delta W^2_{n}
\end{pmatrix}
\\
&\quad-h
\begin{pmatrix}
q_{1,n}+\frac{h}2\Psi_{1,n+1}+\alpha\left( q_{2,n}^2+h\Psi_{2,n+1}q_{2,n}+\frac{h^2}3\Psi_{2,n+1}^2 \right)-\alpha\left( q_{1,n}^2+h\Psi_{1,n+1}q_{1,n}+\frac{h^2}3\Psi_{2,n+1}^2 \right)\\
q_{2,n}+\frac{h}2\Psi_{2,n+1}+2\alpha\left( q_{1,n}q_{2,n}+\frac{h}2q_{1,n}\Psi_{2,n+1}+\frac{h}2q_{2,n}\Psi_{1,n+1}+\frac{h^2}3\Psi_{1,n+1}\Psi_{2,n+1}  \right)
\end{pmatrix}.
\end{align*}

Using $2^{11}$ stepsizes of the drift-preserving scheme \eqref{dp} on the time interval $[0,50]$ 
(using fixed-point iterations for solving the implicit systems), and approximating 
the expected values with $M=10^5$ samples, we obtain 
the result displayed in Figure~\ref{fig:traceHH}. This figure again 
clearly illustrates the excellent long time behaviour of 
the proposed numerical scheme as stated in the above theorem. 

\begin{figure}[h]
\centering
\includegraphics*[height=5cm,keepaspectratio]{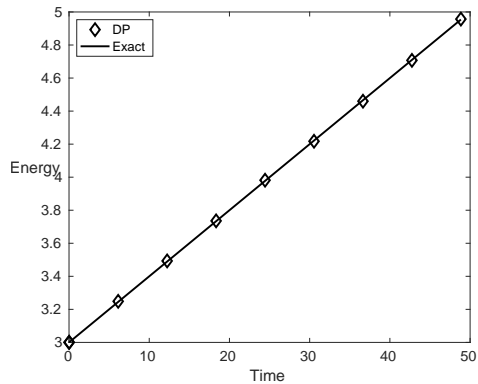}
\caption{Numerical trace formula for the stochastic H\'enon--Heiles problem on $[0,50]$.}
\label{fig:traceHH}
\end{figure}

Furthermore, the drift-preserving scheme \eqref{dp} outperforms the numerical schemes from \cite{bb14,MR3873562} 
in terms of preserving the expected value of the energy (compare Figure~\ref{fig:traceHH} 
to \cite[Table~2]{bb14} and \cite[Figure~6.7]{MR3873562}). 

\section{Acknowledgement}
We appreciate the referees' comments on an earlier version of the paper. 
We would like to thank Gilles Vilmart (University of Geneva) for interesting dicussions. 
The work of CCC is supported by the National Natural Science Foundation of China 
(NOs. 11871068, 11971470, 11711530017, 91630312, 11926417). 
The work of DC was supported by the Swedish Research Council (VR) 
(projects nr. $2013-4562$ and nr. $2018-04443$).
The work of RD was supported by GNCS-INDAM project and by PRIN2017-MIUR project 
"Structure preserving approximation of evolutionary problems". 
The work of AL was supported in part by the Swedish Research Council 
under Reg.~No.~621-2014-3995 and by the Wallenberg AI, Autonomous Systems and Software Program (WASP) 
funded by the Knut and Alice Wallenberg Foundation. 
This work was partially supported 
by STINT and NSFC Joint China-Sweden Mobility programme (project nr. $CH2016-6729$). 
The computations were performed on resources provided by 
the Swedish National Infrastructure for Computing (SNIC) 
at HPC2N, Ume{\aa} University. 

\bibliographystyle{plain}
\bibliography{biblio}

\end{document}